\documentclass[11pt]{article}
\usepackage[margin=1in]{geometry}
\usepackage{akwon}
\usepackage{url}
\usepackage{enumerate,hyperref,mathtools,tikz-cd}
\theoremstyle{plain}
\newtheorem*{theorem*}{Theorem}
\newcommand{\on}{\operatorname}
\newcommand{\up}{\prescript{*}{}}
\newcommand{\uup}{\prescript{**}{}}
\title{A note on Diophantine subsets of large fields}
\author{Andrew Kwon}
\date{}
\begin{document}
\maketitle
\begin{abstract}
	Large fields (also called ample, anti-mordellic) generalize many fields of classical interest, such as algebraically closed fields, real-closed fields, and $p$-adic fields. In this note we generalize a result of Fehm and prove that finite unions of affine translates of proper subfields are never diophantine subsets of perfect large fields.
\end{abstract}
\section{Introduction}
Recall the definition of a large field:\\

\noindent \textbf{Definition.} A field $k$ is \emph{large} if every curve $C$ defined over $k$ with at least one smooth $k$-point has infinitely many $k$-rational points.\\

Such fields contrast starkly with some of the usual fields of arithmetic interest in view of, say, the Mordell Conjecture. Nonetheless, there are many examples of large fields (e.g., henselian fields, pseudo-algebraically closed fields, etc.), and large fields have proven to be rich in their properties and usefulness for Galois theory. For example, the Shafarevich Conjecture on the structure of $\on{Gal}(\bar{\Q}/\Q^{\text{ab}})$ would follow from knowing that $\Q^{\text{ab}}$ is large \cite{Po96}. There have also been recent interesting developments in model theory where large fields are a main player \cite{JTWY}. See \cite{BF13, Po14} and references therein for much more about large fields.\\

Harbater \cite{Ha09} first posed questions about cardinality and properties of rational points over large fields for the study of quasi-free profinite groups. An argument due to Pop shows that for large fields $k$, $C(k) \neq \varnothing$ implies $C(k)$ is actually as large as possible, namely $|C(k)| = |k|$. This was used by Fehm and Petersen \cite{FP21} to conclude that if $k$ is a large field, not algebraic over a finite field, and $A/k$ is a non-zero abelian variety, then the rank of $A(k)$ equals the cardinality of $k$. (As noted by Pop in \cite{Po14}, by far the hardest part is showing that the rank is infinite.)\\

The focus of this note is on diophantine subsets of large fields. Characterizing diophantine subsets is difficult but can have substantial applications. A complete characterization of diophantine subsets of $\Z$ resulted in the completion of Hilbert's Tenth Problem, although neither result is known for general rings of integers. Koll\'ar was the first to show that the large hypothesis on an uncountable field $k$ of characteristic 0 has interesting consequences for the diophantine subsets of function fields over $k$ \cite{Ko08}. For example, if $k(t) \subset K_{1} \subset K_{2}$ are finite extensions and $K_{1}$ is a diophantine subset of $K_{2}$, then $K_{1} = K_{2}$. We instead place our attention on diophantine subsets of the field $k$ itself, and aim to generalize the following analogous result.

\begin{theorem*} (\cite{Fe10}, Theorem 2) Let $k$ be a perfect large field and $\Sigma \subset k$ be an infinite diophantine subset. Then for every proper subfield $k_{0} \subset k$, $|\Sigma \setminus k_{0}| = |k|$.
\end{theorem*}

With a refinement of Fehm's methods, we deduce that many more general subsets have \emph{too much structure} to be diophantine in $k$:

\begin{theorem}\label{thm:main result} Let $k$, $\Sigma$ be as above, $\{k_{i}\}_{i \in I}$ be any finite collection of proper subfields of $k$, and $\{\alpha_{i}\}_{i \in I}, \{\beta_{i}\}_{i \in I}$ be subsets of $k$. Then, 
	\[
	\left|\Sigma \setminus \bigcup_{i \in I} (\alpha_{i} k_{i} + \beta_{i})\right| = |k|.
	\]
\end{theorem}

This implies that a finite union of proper subfields cannot be diophantine. The essence of the proof is to show that $\Sigma \setminus \bigcup_{i \in I} k_{i}$ cannot be empty, and then cardinality considerations will ensure that this set is actually as large as possible. The key for us will essentially be to relate a hypothetical set-theoretic cover $\Sigma \subset \bigcup_{i \in I} k_{i}$ to a group-theoretic cover of $(k,+)$ by subgroups that will be impossible by virtue of the following lemma \cite{Ne54}.

\begin{lemma} \label{lem:Neumann} Let $G$ be any group. Suppose there are distinct proper subgroups $H_{i}$, $i= 1, \ldots, n$ and finite sets of coset representatives $C_{i}$ for $H_{i}$ such that $G = \bigcup_{i=1}^{n} C_{i} H_{i}$. Then, at least one of the $H_{i}$ must have finite index in $G$.
\end{lemma}

\noindent \textit{Proof.} We induct on $n$. Evidently if $n=1$ then $H_{1}$ has finite index in $G$. In general, if $C_{n} H_{n} = G$ then we are done for the same reason; otherwise, there is some $h \not\in C_{n} H_{n}$, with $h H_{n}$ a coset disjoint from $C_{n} H_{n}$. Thus, 
\[
	h H_{n} \subset \bigcup_{i=1}^{n-1} C_{i} H_{i} \Leftrightarrow H_{n} \subset \bigcup_{i=1}^{n-1} h^{-1} C_{i} H_{i},
\]
and we may replace $C_{n} H_{n}$ with the larger collection of cosets $\bigcup_{i=1}^{n-1} C_{n} h^{-1} C_{i} H_{i}$. This reduces the total number of distinct subgroups, and the claim follows inductively.\qed\\

In the interest of keeping the paper somewhat self-contained, the basics of model theory that we will need, namely saturated ultrapowers, are summarized in section 2. In section 3, we show how smooth morphisms yield ``big'' diophantine sets. The proof of the main result is finished in section 4, and in section 5 we discuss some further directions and technical obstacles.

\section{Model-theoretic background}
We take a quick and dirty approach to collect all of the basic facts about (saturated) ultrapowers that we will need, ignoring the more general settings for saturated models or ultraproducts. For details, see any introductory text on model theory, e.g. \cite{CK90}. Let $k$ be a field, $J$ any index set, and $\mathcal{U}$ a non-principal ultrafilter on $J$; then, one can define an equivalence relation on $k^{J}$ via
\[
	(a_{j})_{j} \sim_{\mathcal{U}} (b_{j})_{j} \Leftrightarrow \{j \, : \, a_{j} = b_{j}\} \in \mathcal{U}.
\]

\noindent \textbf{Definition.} The \emph{ultrapower} $\up{k}$ with respect to $J$, $\mathcal{U}$ is defined to be $k^{J} / \sim_{\mathcal{U}}$.\\

\noindent \textbf{Comments.} It is a standard fact that this construction actually yields a field, and $k$ embeds into $\up{k}$ via the diagonal embedding. While the notation $\up{k}$ makes no reference to $J, \mathcal{U}$, there will be no ambiguity in our setting. It may also be worth noting that it is more typical in the literature to write $k^{J}/\mathcal{U}$ than $k^{J}/\sim_{\mathcal{U}}$.\\

Among the most remarkable and important properties about ultrapowers is \L{}os' theorem, a special case of which is the following:

\begin{theorem}
	Any first order formula (in the language of fields) is true in $k$ if and only if it is true in $\up{k}$.
\end{theorem}

From here, we can deduce the following elementary properties of ultrapowers we will use later.

\begin{lemma} \label{lem:basics}
	Let $k$ be a field and $\up{k}$ an ultrapower of $k$. The following hold.
	\begin{enumerate}[(i)]
		\item $k$ is finite $\Leftrightarrow \up{k}$ is finite.
		\item If $k_{i} \subset k$ is a proper subfield, then $\up{k_{i}} \cap k = k_{i}$.
	\end{enumerate}
\end{lemma}

\noindent \textit{Proof.} To (i). This follows from \L{}os' theorem, since for a fixed natural number $n$, the sentence ``there are at most $n$ distinct elements'' is expressible in first order.\\
To (ii). From the construction of the ultrapower, we see that the elements of $\up{k_{i}}$ are equivalence classes of functions on $J$ with values in $k_{i}$. On the other hand, the elements of $k$ (thought of as a subset of $\up{k}$) are equivalence classes of constant functions. Thus, their intersection consists of equivalence classes of constant functions with values in $k_{i}$, i.e., $k_{i} \subset \up{k}$.\qed\\

Also very useful and important are saturated ultrapowers; these are a particular kind of ultrapower that is in some sense ``large enough'' so that many formulas we will describe later can be satisfied simultaneously. We will not recall all of the basic definitions necessary to state things precisely, as the existence of saturated ultrapowers depends on a somewhat subtle analysis of properties of ultrafilters; rather we give an overview for the sake of a reader that may be less familiar with model theory.\\

\noindent \textbf{Definition.} An ultrapower $\up{k}$ is \emph{$\kappa$-saturated} if, for every subset $A$ of $\up{k}$ with $|A| < \kappa$, $\up{k}$ realizes every complete type over $A$.\\

In our setting, a (complete) type over $A$ is essentially a set of first-order formulas with constants coming from $A$ and a finite number of free variables, where every finite subset of formulas is satisfiable in $\up{k}$ after some assignment of the free variables. (Note, the finite satisfiability condition is intrinsic to the definition of a complete type.) Then, to say that $\up{k}$ \emph{realizes} the complete type means that all of the formulas can actually be satisfied simultaneously. The following special case of a theorem due to Keisler ensures that sufficiently saturated ultrapowers of fields exist (particularly because the language of fields is countable).

\begin{proposition} \label{prop:saturated ultrapower existence} (\cite{Ke10}, Theorem 10.5) For $J$ of infinite cardinality $\kappa$, there exists an ultrafilter $\mathcal{U}$ on $J$ such that $k^{J}/\mathcal{U}$ is $\kappa^{+}$-saturated. 
\end{proposition}

In the last section, we will need one more model-theoretic lemma which relates the covering of $\Sigma$ by the $\alpha_{i} k_{i} + \beta_{i}$ to that of $\up{\Sigma}$ by the $\alpha_{i} \up{k_{i}} + \beta_{i}$. Here we give a hands-on proof.

\begin{lemma} \label{lem:extending to ultrapowers} Fix a field $k$ with a finite collection of proper subsets $A_{i} \subset k$, $\up{k} = k^{J}/\mathcal{U}$ an ultrapower of $k$, and $\Sigma \subset k$ be any subset. If $\Sigma \subset \bigcup_{i \in I} A_{i}$, then $\up{\Sigma} \subset \bigcup_{i \in I} \up{A_{i}}$.
\end{lemma}

\textit{Proof.} Pick any element $[(x_{j})_{j}] \in \up{\Sigma}$, where for each $j \in J$, $x_{j}$ is an element of $\Sigma$. By assumption, we can partition $J$ into $|I|$ disjoint subsets $J_{i}$ such that $x_{j} \in A_{i}$ for each $j \in J_{i}$. (If $x_{j}$ is in $A_{i}$ and $A_{i'}$, then an assignment can be made arbitrarily.) However, it is easy to verify from the definition of ultrafilters that at least one of the $J_{i}$ is an element of $\mathcal{U}$, which means that $[(x_{j})_{j}] \in \up{A_{i}}$. Furthermore, it is easy to verify that this is independent of the choice of representative. \qed\\

\section{Images of smooth morphisms}

Now we briefly rephrase an argument from the proof of Theorem 3.1 of \cite{Po14} that allows us to find some arithmetic structure in diophantine sets that arise from smooth morphisms.

\begin{proposition}\label{prop:smooth case} Suppose $\varnothing \neq \Sigma \subset k = \mathbb{A}^{1}(k)$ is the image of $\pi(X(k))$ for some smooth morphism $\pi : X \to \mathbb{A}^{1}$. Then, for any finite subset $A \subset k$, and any $\lambda \in \Sigma$, there are infinitely many $\nu \in \Sigma$ such that 
\[
	A \subset \frac{1}{\nu - \lambda} (\Sigma - \lambda).
\]
\end{proposition}

For a more explicit proof in coordinates, see loc. cit.

\begin{proof}
	Let $\pi_{\lambda}$ denote the composition of $\pi$ with the translation by $-\lambda$, in particular so that $0$ is in the image of $\pi_{\lambda}(X(k))$, say $\pi_{\lambda}(x) = 0$. We first demonstrate how to get a single element of $A$ as a ratio of two elements of $\Sigma - \lambda$, but we disregard $0$ because it will always be contained in $\Sigma - \lambda$ by construction.\\

	Consider the following pullback diagram, where $a$ is some nonzero element of $A$:
\[\begin{tikzcd}
		{X_{\{a\}}} && X \\
	X & {\mathbb{A}^1} & {\mathbb{A}^1}
	\arrow["\pi_{\lambda}"', from=2-1, to=2-2]
	\arrow["\pi_{\lambda}", from=1-3, to=2-3]
	\arrow["{\cdot 1/a}"', from=2-2, to=2-3]
	\arrow["p_{0}", from=1-1, to=1-3]
	\arrow["p_{a}", from=1-1, to=2-1]
\end{tikzcd}\]
Evidently the map $\pi_{\lambda}$ is smooth, hence $X_{\{a\}} \to X$ is smooth and so $X_{\{a\}} \to \on{Spec} k$ is smooth. Furthermore, $X_{\{a\}}$ has $k$-rational points because $x : \on{Spec} k \to X$ factors through $X_{\{a\}}$. Since $k$ is large, we conclude that $X_{\{a\}}$ must have infinitely many $k$-rational points. For all but finitely many $\tilde{x} \in X_{\{a\}}(k)$, if $x_{i} = p_{i}(\tilde{x})$ for $i \in \{0,a\}$, then $\pi_{\lambda}(x_{a}) / a = \pi_{\lambda}(x_{0}) \neq 0$, hence $a = \pi_{\lambda}(x_{a}) / \pi_{\lambda}(x_{0})$.\\

To realize all of $A$ in the quotient of $\Sigma - \lambda$ by a fixed element, we vary $a$ in the above construction and repeatedly take pullbacks along morphisms of the form $(1/a \circ \pi_{\lambda}) : X \to \mathbb{A}^{1} \to \mathbb{A}^{1}$. For example, for two nonzero elements $a_{1}, a_{2} \in A$, we construct $X_{\{a_{1}, a_{2}\}}$ via
\[\begin{tikzcd}
	{X_{\{a_1,a_2\}}} && {X_{\{a_2\}}} \\
	& {X_{\{a_1\}}} && X \\
	&& X \\
	& X && {\mathbb{A}^1}
	\arrow[from=1-1, to=1-3]
	\arrow[from=1-1, to=2-2]
	\arrow[from=1-3, to=2-4]
	\arrow["{p_{a_{2}}}", from=1-3, to=3-3, near start]
	\arrow[from=2-2, to=2-4]
	\arrow["{p_{a_{1}}}"', from=2-2, to=4-2]
	\arrow[from=2-4, to=4-4]
	\arrow["{(1/a_2 \circ \pi_\lambda)}", from=3-3, to=4-4, near start,swap]
	\arrow["{(1/a_1 \circ \pi_\lambda)}"', from=4-2, to=4-4, near start]
\end{tikzcd}\]

In general, we find a smooth variety $X_{A}$ with infinitely many $k$-rational points such that 
\[\begin{tikzcd}
		{X_{A}} && X \\
	X & {\mathbb{A}^1} & {\mathbb{A}^1}
	\arrow["\pi_{\lambda}"', from=2-1, to=2-2]
	\arrow["\pi_{\lambda}", from=1-3, to=2-3]
	\arrow["{\cdot 1/a}"', from=2-2, to=2-3]
	\arrow[from=1-1, to=1-3]
	\arrow[from=1-1, to=2-1]
\end{tikzcd}\]
commutes for every $a \in A$. Therefore, all but finitely many of the points in $X_{A}(k)$ will produce points $x_{0}, \{x_{a}\}_{a \in A}$ in $X(k)$ such that $a = \pi_{\lambda}(x_{a}) / \pi_{\lambda}(x_{0})$ for all $a \in A$, as desired. We take $\nu = \pi(x_{0})$ to get the claimed result.
\end{proof}

\section{The main reduction}
We now present the main reduction steps, where we use (part of) a lemma from \cite{Fe10} that shows an ultrapower $\up{\Sigma}$ ``contains'' the image of a smooth morphism when $k$ is perfect, and then demonstrate that an even more highly saturated ultrapower of $\Sigma$ must therefore contain $k$ (up to some scaling). 

\begin{lemma}[\cite{Fe10}, Lemma 8]\label{lem:finding a smooth morphism} 
	Let $k$ be a perfect large field and $k_{1} \subset k$ a subfield. If $\Sigma$ is existentially $k_{1}$-definable and $\Sigma \not\subseteq \bar{k}_{1}$, then there is a smooth morphism $f : C \to \mathbb{A}^{1}$ with $\varnothing \neq f(C(k)) \subset \Sigma$.
\end{lemma}

The way we may apply the above lemma is by passing to an $\aleph_{1}$-saturated ultrapower of $k$, so that $\up{\Sigma}$ will be uncountable and in particular not contained in the algebraic closure of its field of definition (i.e., the field of definition of $\pi : X \to \mathbb{A}^{1}$), which is countable. Thus, we get the following.

\begin{corollary}\label{cor:finite sentence}
	For any infinite diophantine subset $\Sigma \subset k$, there is an ultrapower $\up{k}$ such that for any finite subset $A \subset k$, there exist $\up{\lambda}, \up{\nu} \in \up{\Sigma}$ such that
	\[
		A \subset \frac{1}{\up{\nu} - \up{\lambda}}(\up{\Sigma} - \up{\lambda}).
	\]
\end{corollary}

To make clear how the saturation is used, we rephrase this statement in terms of a (partial) 2-type $\{\varphi_{a}(x_{0},x_{1})\}_{a \in k}$. For each $a \in k$, consider the first order formula 
\[
	\varphi_{a}(x_{0}, x_{1}) : \exists x \in \Sigma, a = \frac{x-x_{0}}{x_{1} - x_{0}},
\]
where $x_{0}, x_{1}$ are free variables. The above corollary implies the finite satisfiability of the $\varphi_{a}$ in $\up{k}$ by assigning $x_{0} = \lambda$, $x_{1} = \nu$ corresponding to the finite set $A \subset k$. Thus, if we choose another $|k|^{+}$-saturated ultrapower $\uup{k}$ of $\up{k}$ (via \hyperref[prop:saturated ultrapower existence]{Proposition 2.3}), all of the $\varphi_{a}$ will be satisfiable simultaneously in $\uup{k}$, i.e., there exist $\uup{\lambda}, \uup{\nu} \in \uup{k}$ such that 
\[
	k \subset \frac{1}{\uup{\nu} - \uup{\lambda}}(\uup{\Sigma} - \uup{\lambda}).
\]
Note that the formula defining $\Sigma$ in $k$ will define $\up{\Sigma}$ (resp. $\uup{\Sigma}$) in $\up{k}$ (resp. $\uup{k}$). We are now ready to prove the penultimate result.

\begin{proposition} \label{prop:main tool} In the notation above, $\Sigma \setminus \bigcup_{i\in I} (\alpha_{i} k_{i} + \beta_{i})$ is always non-empty.
\end{proposition}

\begin{proof} Suppose for the sake of contradiction that $\Sigma \subset \bigcup_{i \in I} (\alpha_{i} k_{i} + \beta_{i})$. Choosing $\up{k}$, $\uup{k}$ as above, we note that $\uup{(\alpha_{i} k_{i} + \beta_{i})} = \alpha_{i} \uup{k_{i}} + \beta_{i}$ as subsets of $\uup{k}$. Then, we deduce from \hyperref[lem:extending to ultrapowers]{Lemma 2.4} that $\uup{\Sigma} \subset \bigcup_{i \in I} \alpha_{i} \uup{k_{i}} + \beta_{i}$. Therefore,
	\[
		k \subset \frac{1}{\uup{\nu} - \uup{\lambda}}( \uup{\Sigma} - \uup{\lambda}) \subset \bigcup_{i \in I} \frac{1}{\uup{\nu} - \uup{\lambda}} (\alpha_{i} \uup{k_{i}} + \beta_{i} - \uup{\lambda}).
	\]
	\noindent At last, the group theory comes into play. For ease of notation, let $V_{i} = (\frac{\alpha_{i}}{\uup{\nu} - \uup{\lambda}} \uup{k_{i}}) \cap k$ and $\tau_{i}= \frac{\beta_{i} - \uup{\lambda}}{\uup{\nu} - \uup{\lambda}}$. Without loss of generality, we may replace $\tau_{i}$ with another coset representative that is actually in $k$. Then, 
\[
	k = \bigcup_{i \in I} V_{i} + \tau_{i}, 
\]
and we claim that all the $V_{i}$ are proper additive subgroups of $k$. Supposing to the contrary that $\frac{\alpha_{i}}{\uup{\nu} - \uup{\lambda}} \uup{k_{i}} \supset k$, then in particular $1 \in \frac{\alpha_{i}}{\uup{\nu} - \uup{\lambda}} \uup{k_{i}}$, hence $\frac{\alpha_{i}}{\uup{\nu} - \uup{\lambda}} \uup{k_{i}} = \uup{k_{i}}$. However, this implies that $\uup{k_{i}}\supset k$, while $\uup{k_{i}} \cap k = k_{i}$ by \hyperref[lem:basics]{Lemma 2.2}. Thus, the $V_{i}$ are all proper additive subgroups of $k$, and we may apply the \hyperref[lem:Neumann]{Neumann lemma} to this cover to deduce that there is a $V_{i}$ with finite index in $k$ as \emph{groups}; however, note also that $V_{i}$ and $k$ are vector spaces over $k_{i}$. Thus, $k / V_{i}$ is a non-trivial finite abelian group, but also a $k_{i}$-vector space. This is impossible if $k_{i}$ is infinite; yet on the other hand, if $k_{i}$ is finite, then so is $\uup{k_{i}}$, and so the finite subgroup $V_{i}$ cannot have finite index in $k$, which is evidently infinite. We conclude that $\Sigma \subset \bigcup_{i \in I}\alpha_{i} k_{i} + \beta_{i}$ is impossible, as desired.
\end{proof}

At this point, the proof of \hyperref[thm:main result]{Theorem 1.1} is short.\\

\noindent \textit{Proof of Theorem 1.1.} In the case that $k$ is countable, we are guaranteed the existence of some element $\sigma \in \Sigma \setminus \bigcup_{i \in I} (\alpha_{i}k_{i} + \beta_{i})$. It is well-known that finitely generated fields are never large, hence the subfield generated by $\sigma$ will be strictly smaller than $k$. Thus, we may add that subfield to our set $\{k_{i}\}_{i \in I}$, and \hyperref[prop:main tool]{Proposition 4.3} ensures that we can find another element of $\Sigma$, different from $\sigma$, that is outside of $\bigcup_{i \in I} (\alpha_{i} k_{i} + \beta_{i})$. Iterating this demonstrates that $|\Sigma \setminus \bigcup_{i \in I} (\alpha_{i} k_{i} + \beta_{i})|$ is countably infinite.\\

If $k$ is not countable, suppose for the sake of contradiction that $|\Sigma \setminus \bigcup_{i \in I} (\alpha_{i} k_{i} + \beta_{i})| < |k|$. Then, we may similarly consider the subfield $F$ of $k$ generated by all the elements in $\Sigma \setminus \bigcup_{i \in I} (\alpha_{i} k_{i} + \beta_{i})$ over the prime field. Evidently, $|F| < |k|$ as well, and in particular must be a proper subfield. However, then we find that 
\[
	\Sigma \subset \bigcup_{i \in I} (\alpha_{i} k_{i} + \beta_{i}) \cup F,
\]
which is also impossible by \hyperref[prop:main tool]{Proposition 4.3}. We conclude that $|\Sigma \setminus \bigcup_{i \in I} (\alpha_{i} k_{i} + \beta_{i})| = |k|$.\qed

\section{Final comments and questions}
At several points in this argument, it was essential that $I$ is finite; for example, \hyperref[lem:extending to ultrapowers]{Lemma 2.4} and \hyperref[lem:Neumann]{Lemma 1.2} both require this assumption, and one may attempt to remove it. The former case concerns ultrafilters that are closed under countable intersection, also referred to as $\sigma$-complete ultrafilters, and the existence of such ultrafilters is already related to subtle set-theoretic questions (namely the existence of measurable cardinals) and cannot be proven in ZFC.\\

In the latter case, it could anyways be of independent interest to formulate a Neumann lemma that addresses infinitely many subgroups, where there are some obvious constraints; for example, if the family of subgroups of $G$ are indexed by $I$, then one could reasonably require $|I| < |G|$ and ask for a subgroup in the family of index less than $|G|$. Nonetheless, there exist counterexamples to some simple generalizations of the Neumann lemma \cite{JDH}.\\

These indicate that there are some substantial obstacles to removing the finiteness hypothesis on $I$. One can instead try to remove the perfect hypothesis on the field $k$, and in this direction the precedent is set by Anscombe \cite{An19}. There, it is proven for an imperfect large field of characteristc $p$ that if an infinite diophantine subset $\Sigma$ is contained in a subfield $k_{0} \subset k$, then $k_{0} \supset k^{p^{n}}$ for some $n$. The key difficulty in adapting the arguments of this note is in ``finding a smooth morphism'' whose image on $k$-points is contained in $\Sigma$, while Anscombe manages to find such a morphism whose image is contained in $\F_{p}(\up{\Sigma})^{(p^{\infty})}$, the largest perfect subfield of $\F_{p}(\up{\Sigma})$.




\section*{Acknowledgements}
I am indebted to my advisor Florian Pop for his guidance, patience, and encouragement. I would also like to thank Irfan Alam and Eben Blaisdell, who helped me to see that nonstandard objects and methods are not so strange. This material is based upon work supported by the National Science Foundation Graduate Research Fellowship DGE-2236662.\\

\end{document}